\documentclass[11pt]{article}
\usepackage[utf8]{inputenc}
\usepackage{lmodern}
\usepackage{subfiles}
\usepackage{enumitem}
	\setenumerate{label={\normalfont(\roman*)}, itemsep=0em} 
        
\usepackage{amsfonts}
\usepackage{amsthm}
\usepackage{amsmath}
\usepackage{amssymb}
\usepackage{amscd}
\usepackage{thmtools, thm-restate}
\usepackage{mathrsfs}
\usepackage{mathtools}
\usepackage{bm}
\usepackage{esint}

\usepackage[margin=3cm]{geometry}
\usepackage{indentfirst}
\usepackage{graphicx}
\usepackage{graphics}
\usepackage{lscape}
\usepackage{tikz-cd}
\usepackage{tikz,pgf}
	\usetikzlibrary{arrows}
\usepackage{color}
\usepackage{pict2e}
\usepackage{epic}
\usepackage{epstopdf}
\usepackage{titlesec, titlefoot}
	\titleformat{\section}[block]{\Large\bfseries\filcenter}{\thesection}{1em}{}
\usepackage{commath}
\usepackage{float}
\usepackage{caption}
\usepackage{etoolbox}
\usepackage[affil-it]{authblk}
\usepackage{combelow}
\usepackage{bbm}

\usepackage[hidelinks]{hyperref}
\hypersetup{bookmarksopen=true} 
\usepackage{hypcap}

\newcommand{\R}{\mathbb{R}}

\newcommand{\Id}{\operatorname{Id}}

\newcommand{\spt}{\operatorname{spt}}

\newcommand{\WW}{\operatorname{W}}
\newcommand{\LL}{\mathrm{L}}

\newcommand{\X}{\mathbb{X}}
\newcommand{\Y}{\mathbb{Y}}

\def\loc{\operatorname{loc}}

\def\eqn#1$$#2$${\begin{equation}\label#1#2\end{equation}}

\newcommand{\p}{\partial}

\newcommand{\e}{\varepsilon}

\newcommand{\Der}{\nabla}
\newcommand{\Diff}{\mathrm D}

\def\XXint#1#2#3{{\setbox0=\hbox{$#1{#2#3}{\int}$}
     \vcenter{\hbox{$#2#3$}}\kern-.5\wd0}}

\newcommand{\Langle}{\left\langle}
\newcommand{\Rangle}{\right \rangle}
\newcommand{\defeq}{\vcentcolon=}
\newcommand{\eqdef}{= \vcentcolon}

\newcommand{\lb}{\lvert}
\newcommand{\rb}{\rvert}
\newcommand{\Ll}{\mathrm{L}}
\newcommand{\Ww}{\mathrm{W}}
\newcommand{\Bb}{\mathrm{B}}

\theoremstyle{plain}
\newtheorem{thm}{Theorem}
\newtheorem{example}[thm]{Example}
\newtheorem{lemma}[thm]{Lemma}

\newtheorem{corollary}{Corollary}
\newtheorem{theorem}{Theorem}
\newtheorem{proposition}{Proposition}
\newtheorem{remark}{Remark}

\title{A note on improved differentiability for the Banach-space valued Finsler \texorpdfstring{$\gamma$}{}-Laplacian}
\author{Max Goering and Lukas Koch}
\affil{\small MPI for Mathematics in the Sciences, Inselstrasse 22, 04103 Leipzig, Germany \protect \\
{  {\tt{\{goering, lkoch\}@mis.mpg.de}}\vspace{1em}}}

\makeatother
\begin{document}
\maketitle

\begin{abstract}
We obtain improved fractional differentiability of solutions to the Banach-space valued Finsler $\gamma$-Laplacian defined on a $\sigma$-convex, $\tau$-smooth Banach space. The operators we consider are non-linear and very degenerately elliptic. Our results are new already in the $\R$-valued setting.
\end{abstract}

\section{Introduction}
Let $\Omega\subset \R^n$ be an open, bounded domain and let $(\X,\rho)$ be a $\sigma$-convex, $\tau$-smooth Banach space. Consider $u\colon \Omega \to \X$ and let $F$ be a norm on $\X\otimes \R^n$. We study solutions of the problem
\begin{equation} \label{e:fullgen}
\int_{\Omega} \Langle F(\Der u)^{\gamma-1} \Diff F(\Der u), \Der  \varphi \Rangle_{\X \otimes \R^{n}} = \int_{\Omega} \langle f, \varphi \rangle_{\X} \qquad \forall \varphi \in \Ww^{1,\gamma}_{0}(\Omega;\X),\tag{F$\gamma$L}
\end{equation}
for sufficiently regular $f$. Here $\Diff F$ denotes the derivative of the map $F\colon \X\otimes \R^n\to \R$. We refer to \eqref{e:fullgen} as the Banach-space valued Finsler $\gamma$-Laplacian. We refer the reader to Section \ref{s:background} for precise definitions and notation.

Given a norm $\rho$ on $\R^{n}$ the (scalar) Finsler Laplacian 
\begin{equation} \label{e:flap}
\int_{\Omega} \Langle \rho(\Der u) \Diff \rho(\Der u), \Der  \varphi \Rangle = 0 \qquad \forall \varphi \in \Ww^{1,2}_{0}(\Omega;\R)
\end{equation}
is closely related to the anisotropic isoperimetric inequality, 
\cite{belloni2003isoperimetric}. Under additional 2nd order ellipticity conditions, the Finsler Laplacian has also been used to study bubbling phenomena for volume constrained minimizers \cite{delgadino2018bubbling}. Some qualitative regularity for the Finsler Laplacian was first shown in \cite{ferone2009remarks}, where no quantitative regularity or convexity of $\rho$ was assumed.

Recently, extending results of \cite{belloni2003isoperimetric}, the Finsler $\gamma$-Laplacian 
\begin{equation} \label{e:gflap}
\int \Langle \rho(\Der u)^{\gamma-1} \Diff \rho(\Der u), \Der  \varphi \Rangle = 0 \qquad \forall \varphi \in \Ww^{1,\gamma}_{0}(\Omega;\R)
\end{equation}
has been studied in \cite{kawohl2008p,xia2019sharp, di2020eigenvalues,kristaly2022nonlinear} with a focus on spectral properties, due to applications to the isoperimetric inequality and Finslerian analogs of other classical inequalities in Riemannian geometric analysis. Concerning the regularity theory of solutions to \eqref{e:gflap}, in the scalar Euclidean set-up standard theory applies, see \cite[Chapter 7]{giusti2003}, and in particular H\"older regularity of solutions can be recovered via Harnack's inequality and maximum principles, see also \cite{ferone2009remarks} and \cite{goering2020harnack} for further comments in this direction. 

In order to relate well-studied (quantified) geometric properties of Banach spaces to regularity of solutions of the Finsler $\gamma$-Laplacian, we make use of duality methods and notation taken from Banach space geometry, see \cite{xu1991characteristic}. In particular, given a $\sigma$-convex and $\tau$-smooth Banach space $(\Y, F)$, we make use of the duality mappings $j_{F}^{\gamma}(\xi) \defeq F(\xi)^{\gamma-1} \mathrm{D} F(\xi)$. Here we only mention that $j_{F}^{\gamma} (\nabla u)$ coincides with the stress functional, which is of recent interest in the calculus of variations. For the readers convenience, we include background for these tools in Section \ref{s:BanachGeometry}. 

Duality techniques in the context of the calculus of variations already appear in \cite{Zhikov1994} and were used in the context of integrands with linear growth in \cite{Seregin1993}. More recently, such ideas have been employed in the context of linear growth \cite{Koch2022}, standard growth \cite{Carozza2005,Carozza2006,Carozza2013}, faster than exponential growth \cite{Bonfanti2012,Bonfanti2013}, and $(p,q)$-growth \cite{Koch2022a}. To facilitate a clear analysis of how our results relate to the literature on quasi-linear elliptic PDEs, we state our main result precisely. 
\begin{theorem} \label{t:mainintro}
Fix $1 < \gamma < \infty, \tau \in (1,2]$, and $\sigma \in [2,\infty)$. Let $(\X, \rho)$ be a Banach space, $\Omega \subset \R^{n}$ a domain, and $F$ a $\sigma$-convex and $\tau$-smooth norm on $\X \otimes \R^{n}$ satisfying for some norm $\| \cdot \|$ on $\R^{n}$
$$
F(\xi_{1}, \dots, \xi_{n}) \lesssim \|(\rho(\xi_{1}), \dots \rho(\xi_{n}) \|.
$$
Set $\alpha = \max ( \min (\gamma^{\prime}, \sigma^{\prime} ) , \min ( \gamma, \tau))$, and $\alpha_{\ast} = \alpha -1$. Assume $f\in \Bb^{\alpha_\ast,\gamma^\prime}_{\infty}(\Omega)$. Suppose $u \in \Ww^{1,\gamma}(\Omega;\X)$ solves \eqref{e:fullgen}. If 
$$
V(\Der u) \defeq F(\nabla u)^{\frac{\gamma-\sigma}{\sigma}} \nabla u, \quad\text{and}\quad V^{\ast}(j_F^\gamma(\Der u)) \defeq F_\ast( j_{F}^{\gamma}(\nabla u))^{\frac{\gamma^{\prime}-\tau^{\prime}}{\tau^{\prime}}} j_{F}^{\gamma}(\nabla u)
$$
respectively, denote the $V$- and dual $V$-functions for $\nabla u$, then
\begin{equation} \label{e:vfuncIntro}
V(\Der u) \in \Bb^{\frac{\alpha}{\sigma},\sigma}_{\infty,\loc}(\Omega) \quad \text{and} \quad V^{\ast}(j_F^\gamma(\Der u)) \in \Bb^{\frac{\alpha}{\tau^{\prime}}, \tau^{\prime}}_{\infty,\loc}(\Omega).
\end{equation}
In fact, for any $B_{r} \Subset B_{s} \Subset \Omega$,
$$
\|V(\Der u)\|^{\sigma}_{\Bb^{\frac{\alpha}{\sigma},\sigma}_{\infty}(B_{r})}+\|V_{\ast}(j_F^\gamma(\Der u)\|^{\tau^{\prime}}_{\Bb^{\frac{\alpha}{\tau^{\prime}},\tau^{\prime}}_{\infty}(B_{r})} \lesssim_{(s-r)} \|u\|_{\Ww^{1,\gamma}(B_{s})}^\gamma + \|f\|_{\Bb_{\infty}^{\alpha_{\ast},\gamma^\prime}(B_{s})}^{\gamma^\prime} .
$$
Moreover, 
\begin{equation} \label{e:Intro}
\nabla u \in \Bb_{\infty,\loc}^{\frac{\alpha}{\max\{\sigma,\gamma\}},\gamma}(\Omega) \quad \text{and} \quad j_{F}^{\gamma}(\nabla u) \in \Bb_{\infty,\loc}^{\frac{\alpha}{\max\{\tau^{\prime}, \gamma^{\prime}\}},\gamma^{\prime}}(\Omega).
\end{equation}
\end{theorem}
A wealth of examples of functionals $F$ satisfying the assumptions of Theorem \ref{t:mainintro} can be constructed as follows.
\begin{example}\label{ex:BanachSpaces}
When $(\X, \rho)$ is a $\sigma$-convex and $\tau$-smooth Banach space, we can identify $\X \otimes \R^{n}$ with the space of ''matrices'' with $n$ columns, where each column is an elements in $\X$ after some choice of basis. This is analogous to how $\R^{m} \otimes \R^{n}$ is often identified with the space of $m \times n$ matrices. Then, define $F : \X \otimes \R^{n} \to \R$  by setting for $1<p<\infty$,
\begin{equation} \label{e:defF}
F(x_{1}, \dots, x_{n}) = \| (\rho(x_{1}), \dots, \rho(x_{n}) ) \|_{\ell^{p}(\R^{n})} \qquad \forall x_{i} \in \X
\end{equation}
Then $(\X\otimes \R^n,F)$ is $\min(p,\tau)$-smooth and $\max(p,\sigma)$-convex. We prove this fact in Proposition \ref{p:norms} and give examples of $\sigma$-convex and $\tau$-smooth Banach spaces in Example \ref{ex:spaces}.
\end{example}

The $V$- and dual $V$-functions, denoted by $V$, $V^{\ast}$ in Theorem \ref{t:mainintro} are specific choices from the family of $V$ functions associated to a Banach space $(\Y,F)$ and its dual $(\Y^{\ast}, F_{\ast})$, defined by
\begin{equation} \label{e:vfuncdef}
V_{p,q}(\xi) = F(\xi)^{\frac{p-q}{q}} \xi \quad \text{and} \quad V_{p,q}^{\ast}(\xi^{\ast}) = F_{\ast}(\xi^{\ast})^{\frac{p-q}{q}} \xi^{\ast}.
\end{equation}

We choose to study \eqref{e:fullgen} in a general $\sigma$-smooth and $\tau$-convex Banach space $(\X,\rho)$. This poses essentially no additional difficulty compared to working over $\R^m$, but enables the application of our theorem to interesting Banach spaces. The main additional difficulty is to reprove a technical estimate concerning the $V$-functional in this setting. Since this result may be of independent interest in the study of Banach-space valued problems in the calculus of variations, we state it here.
\begin{lemma}\label{lem:VFunc}
Let $(\Y, F)$ be a Banach space. Then, for any $p, q >0$,

\begin{equation*}
F\left( V_{p,q}(\xi) - V_{p,q}(\eta) \right) \sim \left( F( \xi) + F(\eta) \right)^{\frac{p-q}{q}} F(\xi - \eta).
\end{equation*}
The implicit constants depend only on $\frac{p-q}{q}$.
\end{lemma}
While Lemma \ref{lem:VFunc} is well known when $F$ is the Euclidean norm on $\R^n$, our proof is elementary and offers insight even in this case.

Our method of proving Theorem \ref{t:mainintro} follows difference quotient arguments well-known in the study of solutions to elliptic systems with $p$-growth or even so-called non-standard growth, see \cite{giusti2003} for an introduction. The main difficulty is the lack of ellipticity inherent in the Finsler $\gamma$-Laplacian. The Euclidean-norm is $2$-smooth and $2$-convex. Weakening these assumptions to $\tau$-smoothness and $\sigma$-convexity fundamentally changes the ellipticity. In particular very degenerate elliptic behavior is possible (see Example \ref{r:ellipticity}). This also changes the precise form of growth bounds of the system \eqref{e:fullgen}. $\sigma$-convexity and $\tau$-smoothness have been studied since at least \cite{clarkson1936uniformly}. Nonetheless, a direct relationship between the regularity of solutions to any Finsler $\gamma$-Laplacian and the basic quantified geometric properties of the Finslerian norm have not previously been exploited. By using the so-called characteristic inequalities of \cite{xu1991characteristic}, we are able to transfer difference quotient techniques to this setting.

We emphasize that we deal with a broad class of problems and do not impose structure conditions beyond $\sigma$-convexity and $\tau$-smoothness, and a natural condition that ensures the norm $F$ on the space of gradients $\X \otimes \R^{n}$ behaves well relative to the norm $\rho$ on the space of partial derivatives $\X$. Our results also suggest new critical thresholds concerning the regularity of solutions. These thresholds depend on the relationship between the smoothness and convexity of $F$ as well as the homogeneity $\gamma$ of the equation. We give two examples to demonstrate the type of problems we are able to deal with and to contrast our results with the existing literature.

\begin{example} \label{r:ellipticity}

The elliptic behavior of \eqref{e:fullgen} can be badly behaved, even in relatively simple cases, when $\X = \R$ and $\gamma = 2$. Consider $\rho = \| \cdot \|_{\ell^{p}}$ for some $p > 2$. Then, the Finsler Laplacian takes the strong form,
$$
- \textrm{div} \left( \Diff|_{\nabla u(x)} \| \cdot\|_{\ell^p}^{2} \right) = - \textrm{div} \left( A(x) \Der u(x) \right) = 0,
$$
where 
$$
(A(x))^{i}_{j} = \| \Der u(x)\|_{\ell^p}^{2-p} \lvert\partial_{i}u(x)\rvert^{p-2} \delta^{i}_{j}.
$$

Thus when $\Der u(x_{0}) = \lambda e_{i}$ for some standard basis direction, it follows $A(x_{0})$ has precisely $1$-nonzero entry, no matter the size of $\lvert\lambda\rvert = \lvert \Der u(x_{0})\rvert$.  Consequently, \eqref{e:gflap} falls outside the realm of PDEs studied in \cite{colombo2017regularity} and \cite{brasco2010congested}, where elliptic operators with degeneracies in a convex set were studied.
\end{example}

\begin{example} \label{x:rotations}
The orthotropic $p$-Laplacian, arises as the Finsler $p$-Laplacian for the norm $\rho(\cdot)= \| \cdot \|_{\ell^p}$ and takes the weak form
$$
\sum_{i} \int \lvert\partial_{i} u\rvert^{p-2} \partial_{i} u \partial_{i} \varphi = 0. 
$$

This equation has been studied extensively, see \cite{demengel2016lipschitz}, with substantial work performed on generalizations focused on preserving the orthotropic structure, see for instance \cite{brasco2017sobolev}. One of the key techniques to study local regularity of energy minimizing surfaces is to produce graphical approximations, and verify that the approximations almost solve some PDE. When $\| \cdot \|_{p}$-minimizing surfaces are flat with respect to a standard basis direction, see \cite{goering2020harnack} for a more precise discussion, the key PDE to study is the orthotropic $p$-Laplacian. When the surface is flat with respect to other planes, it is necessary to study rotations of the orthotropic $p$-Laplacian, i.e., the Finsler $p$-Laplacian with norm $\rho ( \cdot ) = \| M \cdot \|_{\ell^p}$ for some rotation matrix $M$. This PDE takes the weak form
$$
\sum_{i} \int \lvert (M \Der u)^{i}\rvert^{p-2} (M \Der u)^{i} (M \Der \varphi )^{i} = 0.
$$
Due to the non-linear mixing of partial derivatives, one cannot hope to adapt techniques depending upon the orthotropic structure of the PDE to this setting. Hence improved differentiability for rotations of the orthotropic $p$-Laplacian is an open question answered by Theorem \ref{t:mainintro}. This answer is made explicit in Corollary \ref{t:classical} and the discussion following it.
\end{example}

To aid in comparing Theorem \ref{t:mainintro} to the existing literature, we state a simplified version in the case that $F$ is either $2$-smooth or $2$-convex and simplify to finite-dimensional systems, i.e., the case that $\X = \R^{m}$. 

\begin{corollary}\label{t:classical}
Suppose $u \in \Ww^{1,\gamma}(\Omega, \R^{m})$ solves \eqref{e:fullgen}.
If $F$ is $2$-smooth 
and $\sigma$-convex, then whenever $\gamma \ge 2$,
\begin{equation} \label{e:classicaldual}
V^{\ast}(j_F^\gamma(\Der u)) \in \Ww^{1,2}_{\loc}(\Omega) \quad \text{and} \quad j_{F}^{\gamma}(\nabla u) \in \Ww^{1,\gamma^{\prime}}_{\loc}(\Omega).
\end{equation}
If $F$ is $\tau$-smooth, and $2$-convex, then whenever $\gamma \le 2$, 
\begin{equation} \label{e:classicalprimal}
V(\Der u) \in \Ww^{1,2}_{\loc}(\Omega) \quad \text{and} \quad \nabla u \in \Ww_{\loc}^{1,\gamma}(\Omega).
\end{equation}
\end{corollary}

The dual statement in Corollary \ref{t:classical}, that is \eqref{e:classicaldual}, follows from Theorem \ref{t:mainintro} when $\alpha = \min (\gamma, \tau) = 2$. The primal statement of Corollary \ref{t:classical}, i.e. \eqref{e:classicalprimal}, follows from Theorem \ref{t:mainintro} when $\alpha = \min (\sigma^{\prime}, \gamma^{\prime}) = 2$. 

Corollary \ref{t:classical} recovers a number of classical results. To recover standard results for the $p$-Laplacian from Corollary \ref{t:classical}, note when $F = \lvert \cdot \rvert$ is the Euclidean norm, it is $2$-smooth, $2$-convex, and $\nabla F(z) = \frac{ z }{F(z)}$. So, in this setting,
$$
V(\Der u) = V^{\ast}(j_F^\gamma(\Der u)) = \lvert\nabla u(x)\rvert^{\frac{p-2}{2}} \nabla u(x).
$$
Therefore, the regularity of $V(\Der u)$ when $p \ge 2$ and of $V^{\ast}(j_F^\gamma(\Der u))$ when $p \le 2$ recovers the classical theorem due to Uhlenbeck, that $\lvert\nabla u\rvert^{\frac{p-2}{2}} \nabla u \in W^{1,2}_{\loc}$ when $u$ is $p$-harmonic, proven with difference quotient methods in Euclidean space in \cite{bojarski1987p}.

We can also recover results for the orthotropic $p$-Laplacian. Recall the orthotropic $p$-Laplacian is the Finsler $\gamma$-Laplacian when $F = \| \cdot \|_{p}$ and $\gamma = p$. When $p \ge 2$, $\| \cdot \|_{p}$ is $2$-smooth and $p$-convex.  In this case,
$$
V^\ast(\Der u) = F(j_{F}^{p}(\nabla u))^{\frac{p^{\prime}-2}{2}} j_{F}^{p}( \nabla u) = \|\nabla u \|_{\ell^{p}}^{\frac{2-p}{2}}\sum_{i} \lvert\partial_{i} u\rvert^{p-2} \partial_{i} u e_{i}.
$$
Then, Corollary \ref{t:classical} says that $V^\ast(j_F^\gamma(\Der u)) \in \Ww^{1,2}_{\loc}(\Omega)$  when $p \ge 2$. In \cite[Theorem 1.1.]{brasco2017sobolev} a very similar theorem is proven for generalizations of the orthotropic $p$-Laplacian, which arise by focusing on the anisotropic structure. In fact, for the orthotropic $p$-Laplacian \cite[Theorem 1.1.]{brasco2017sobolev} states that
$$
\mathcal{V}^{\ast}(\nabla u) = \sum_{i=1}^{n} \lvert\partial_{i} u\rvert^{\frac{2-p}{2}} l\vert\partial_{i} u\rvert^{p-2} \partial_{i} u e_{i} \in \Ww^{1,2}_{\loc}(\Omega),
$$
assuming additionally that $u$ is bounded.

When $p \le 2$, $\| \cdot \|_{\ell^{p}}$ is $p^{\prime}$-smooth and $2$-convex. So, when $u$ is orthotropic $p$-harmonic, \eqref{e:classicalprimal} of Corollary \ref{t:classical} states
$$
V(\Der u) = F(\nabla u)^{\frac{2-p^{\prime}}{p^{\prime}}} \nabla u = \|\nabla u \|_{\ell^p}^{\frac{p-2}{2}} \nabla u \in W^{1,2}_{\loc}(\Omega) \quad\text{and}\quad u \in W^{2,p}_{\loc}(\Omega).
$$
We believe experts in the field may not be surprised by this result, and might already know it. But, we have not been able to find an explicit reference in the literature.

Finally, we remark that for any $1 < p < \infty$, Corollary \ref{t:classical} also applies to rotations of the orthotropic $p$-Laplacian, see Example \ref{x:rotations}, which cannot be addressed in the existing literature.

So far we have highlighted the relation of our results to classical regularity statements concerning $\Der u$, $V(\nabla u)$, and $V^{\ast}(j_F^\gamma(\nabla u))$. Recently improved regularity for the stress, which herein is denoted by $j_F^\gamma(\Der u)$, has gained increasing attention. We do not aim to give an overview here, but highlight some related results to which we refer for further references. In \cite{cianchi2018} double-sided global estimates relating $\|\Diff F(\Der u)\|_{\WW^{1,2}(\Omega)}$ and $\|f\|_{\LL^2(\Omega)}$ in the case of elliptic equations with Uhlenbeck structure have been shown. \cite{avelin2018} provides a non-linear Cald\'{e}ron-Zygmund theory in the setting of elliptic equations of $p$-growth with measure-valued right-hand side. The technique heavily relies on stress regularity. Moreover, regularity transfer (in terms of differentiability) from the right-hand side to the stress for the $p$-Laplace operator has been studied in \cite{balci2020}. We note that with regards to the regularity of the right-hand side, our results are weaker than those obtained in the classical setting. It would be interesting to see whether our results can be improved  in order to capture (sharp) assumptions on the right-hand side. Finally, we remark that in \cite{guarnotta2022} the question of what structure conditions on $F$ are necessary in order to obtain $\WW^{1,2}_{loc}$-regularity of the stress  for minima of the functional $\int_\Omega F(\Der u)$ are studied. The key assumption was found to be quasi-conformality of the map $z\to \Diff F(z)$. It would be interesting to study whether such a result extends to our setting.

In Section \ref{s:background}, we provide some background results on Banach space geometry and then prove Lemma \ref{lem:VFunc} in Section \ref{s:VFunc}. We finally prove our main Theorem in Section \ref{s:proof}.

\section{Preliminaries} \label{s:background}
Throughout, given any number $p \in (1,\infty)$, we let $p^{\prime}$ denote its Hölder conjugate. For a given function $u$ on $\Omega$ and $h\in \X$, denote $\Delta_h u(x)=u(x)-u(x+h)$. If $x\not\in\Omega$ or $x+h \not\in \Omega$, we understand $\Delta_{h} u(x) \equiv 0$. We denote by $\lvert \cdot\rvert$ the usual $L^2$-based norm on $\R^n$.

We refer to \cite{Hytonen2016} or \cite{Kreuter2015} for the theory of vector-valued Sobolev and Besov spaces. For our purposes we only say that if $\Omega \subset \R^{n}$ and $\Y$ is a reflexive Banach space and $\alpha\in(0,1)$, the Besov space $\Bb_{\infty}^{\alpha,\gamma}(\Omega;\Y)$ is characterized by difference quotients, i.e.,
\begin{equation} \label{e:dqchar}
u \in \Bb^{\alpha,\gamma}_{\infty}(\Omega;\Y) \iff \exists C > 0 \text{ so that } \| \Delta_{h} u \|_{\LL^{p}(\omega)} \le C |h|^{\alpha} \quad \forall \omega \Subset \Omega.
\end{equation}

Further, $u \in \Bb^{\alpha,\gamma}_{\infty,loc}(\Omega;\Y)$ if $C$ is allowed to depend on $\omega$. If $\alpha=1$, the same definition gives the usual Sobolev space $\WW^{1,p}(\Omega)$ and its local variant $\WW^{1,p}_{loc}(\Omega)$.

\subsection{Banach space geometry}\label{s:BanachGeometry}
Let $(\X, \rho)$ be a real Banach space and $(\X^{\ast}, \rho_{\ast})$ denote its dual. The modulus of convexity
$$
\sigma_{\rho}(\epsilon) = \inf \left\{ 1 - \rho \left( \frac{x+y}{2} \right) ~ : ~ x,y \in \{ \rho = 1\}, \rho(x-y) \ge \epsilon \right\}
$$
and the modulus of smoothness
$$
\tau_{\rho}(\epsilon) = \frac{1}{2} \sup \left\{ \rho(x+y) + \rho(x-y) - 2 ~ : ~ x \in \{ \rho = 1 \}, \rho(y) \le \epsilon \right\}
$$
are fundamental in understanding the geometry of $\X$. 

For $\tau \in (1,2]$ and $\sigma \in [2,\infty)$, we define a Banach space $(\X,\rho)$ to be $\sigma$- convex if
$$
\sigma_{\rho}(\epsilon) \gtrsim \epsilon^{\sigma} \qquad \forall \epsilon \in [0,2]
$$
and $(\X,\rho)$ is called $\tau$-smooth if
$$
\tau_{\rho}(\epsilon) \lesssim \epsilon^{\tau} \qquad \forall \epsilon \in [0, \infty).
$$

For any $\gamma > 1$, the duality mapping $j_{\rho}^{\gamma} : \X \to \X^{\ast}$ defined by
$$
j_{\rho}^{\gamma}(x) \defeq \left\{ x^{\ast} : \rho(x^{\ast}) = \rho(x)^{\gamma-1} \text{ and } \langle x, x^{\ast} \rangle = \rho(x)^{\gamma} \right \}
$$ 
coincides with the single-valued function 
\begin{equation} \label{e:dualmap}
j_{\rho}^{\gamma}(x) = \rho(x)^{\gamma-1} \nabla \rho (x),
\end{equation}
whenever $\X$ is $\tau$-smooth and $\sigma$-convex for some $\tau \in (1,2]$ and some $\sigma \in [2, \infty)$. See for example \cite[Properties (J2,J3)]{xu1991characteristic}. Moreover a Banach space is reflexive if and only if it is uniformly smooth or uniformly convex \cite[Proposition 1.e.3]{lindenstrauss2013classical}. In particular, if $(\Y, F)$ is $\tau$-smooth and $\sigma$-convex, it is reflexive and \eqref{e:dqchar} holds. We recall,

\begin{proposition} \label{p:smoothconvexduality}
A Banach space $(\X, \rho)$ is $\sigma$-convex and $\tau$-smooth if and only if its dual $(\X^{\ast}, \rho_{\ast})$ is $\tau^{\prime}$-convex and $\sigma^{\prime}$-smooth.
\end{proposition}

It is also well-known that if $\tau \in (1,2]$ and $\sigma \in [2,\infty)$ then 
\begin{equation} \label{e:ident}
j_{\rho}^{\gamma} \circ j_{\rho_{\ast}}^{\gamma^{\prime}} = \Id_{\X_{\ast}} \quad \text{and} \quad j_{\rho_{\ast}}^{\gamma^{\prime}} \circ j_{\rho}^{\gamma} = \Id_{\X},
\end{equation}
see for instance \cite{bargetz2020rate}.

\begin{example}\label{ex:spaces} We note the following examples of $\sigma$-smooth and $\tau$-convex spaces.
\begin{enumerate}
\item $\|\cdot\|_{\ell^p}$ is $\max(p,2)$-convex and $\min(p,2)$-smooth for $p\in(1,\infty)$.
\item $\|\cdot\|_{\Ll^p}$ is  $\max(p,2)$-convex and $\min(p,2)$-smooth for $p\in(1,\infty)$.
\item Note that for $p\in(1,\infty)$, $\Ww^{m,p}(\Omega,\Y^n)$, equipped with the usual $L^1$-based norm is not $\tau$-smooth or $\sigma$-convex. However, due to Proposition \ref{p:norms}, if endowed with the equivalent norm
\begin{align*}
\|u\|_{\Ww^{m,p}(\Omega,\Y)} = \left(\sum_{\lvert\alpha\rvert\leq m} \|\nabla^\alpha u\|_{\Ll^p(\Omega,\Y)}^2\right)^\frac 1 2
\end{align*}
it is $\max(p,2)$-convex and $\min(p,2)$-smooth. So, one must be careful since equivalent norms do not preserve smoothness and convexity of Banach spaces.
\end{enumerate}
\end{example}

We now recall the main results of \cite{xu1991characteristic} which are relevant to this article, restated
in the special case of $\sigma$-convex and $\tau$-smooth Banach spaces. 
\begin{theorem} \label{t:xuroach}
Let $(\X,\rho)$ be a Banach space. The following are equivalent:
\begin{enumerate}
\item $\rho$ is $\sigma$-convex 
\item For all $\gamma > 1$ and every $x,y \in \X$
\begin{equation} \label{e:sigmaconvex}
\langle j_{\rho}^{\gamma}(x) - j_{\rho}^{\gamma}(y) , x-y \rangle \gtrsim_{\gamma} ( \rho(x) + \rho(y) )^{\gamma-\sigma} \rho(x-y)^{\sigma}.
\end{equation}
\item There exists a $\gamma > 1$ so that \eqref{e:sigmaconvex} holds.
\end{enumerate}

Alternatively, the following are equivalent:
\begin{enumerate}
\item[(a)] $\rho$ is $\tau$-smooth 
\item[(b)] For all $\gamma > 1$ and every $x,y \in \X$
\begin{equation} \label{e:tausmooth}
\rho_{\ast} \left( j_{\rho}^{\gamma}(x) - j_{\rho}^{\gamma}(y) \right) \lesssim_{\gamma} \left( \rho(x) + \rho(y) \right)^{\gamma-\tau} \rho(x-y)^{\tau-1} .
\end{equation}
\item[(c)] There exists a $\gamma > 1$ so that \eqref{e:tausmooth} holds.
\end{enumerate}
\end{theorem}

\begin{remark}
In light of Proposition \ref{p:smoothconvexduality} and Theorem \ref{t:xuroach}, we note the following dual characterization of $\sigma$-convexity and $\tau$-smoothness of $(\X,\rho)$. Namely, $\rho$ is $\tau$-smooth if and only if for every $\gamma >1$
\begin{equation} \label{e:tauprimeconvex}
\left( \rho_{\ast}(x^{\ast}) + \rho_{\ast}(y^{\ast}) \right)^{\gamma^{\prime} - \tau^{\prime}} \rho_{\ast} \left( x^{\ast}-y^{\ast}\right)^{\tau^{\prime}} \lesssim_{\gamma} \Langle j_{\rho_{\ast}}^{\gamma^{\prime}}(x^{\ast}) - j_{\rho_{\ast}}^{\gamma^{\prime}}(y^{\ast}), x^{\ast} - y^{\ast} \Rangle \qquad \forall x^{\ast}, y^{\ast} \in \X^{\ast},
\end{equation}
while $\rho$ is $\sigma$-smooth if and only if for every $\gamma > 1$,
\begin{equation} \label{e:sigmaprimesmooth}
\rho \left( j_{\rho_{\ast}}^{\gamma^{\prime}}(x^{\ast}) - j_{\rho^{\ast}}^{\gamma^{\prime}}(y^{\ast}) \right) \lesssim \left( \rho_{\ast}(x^{\ast}) + \rho_{\ast}(y^{\ast}) \right)^{\gamma^\prime-\sigma^{\prime}} \rho_{\ast} \left( x^{\ast} - y^{\ast} \right)^{\sigma^{\prime} -1}.
\end{equation}
\end{remark}

We now demonstrate a straightforward relationship between $2$-smoothness and $2$-convexity and more standard notions of smoothness and ellipticity in the calculus of variations.
\begin{lemma}\label{lem:C2} Suppose $F\in C^{2}(\R^m\otimes \R^n \setminus \{0\},\R)$ is a norm on $\R^m\otimes \R^n$. Then $(\R^m\otimes \R^n,F)$ is $2$-smooth. On the other hand, if $\rho \in C^{1}(\R^{m} \otimes \R^{n} \setminus \{0\}, \R)$ is a norm and $F = \rho^{2}$ is uniformly elliptic in the sense that
\begin{equation} \label{e:unifelliptic}
\langle \nabla F(x) - \nabla F(y) , x-y \rangle \gtrsim \lvert x-y\rvert^{2},
\end{equation}
then $\rho$ is $2$-convex.
\end{lemma}
\begin{proof}
First suppose $F \in C^{2}(\R^{m} \otimes \R^{n} \setminus \{0\}, \R)$ is a norm. Let $x\in\{F=1\}$ and $F(y) < 1/2$. Then $F(x-y), F(x+y) \in [1/2,3/2]$. From the regularity of $F$ and the fundamental theorem of calculus,
\begin{align*}
F(x+y)+F(x-y)-2 =& F(x+y)-F(x)+F(x-y)-F(x)\\
=& \int_0^1 \langle\mathrm{D}  F(x+t y)-\mathrm{D} F(x-t y), y\rangle\mathrm{d}t\\
= & 2\int_0^1\int_0^1 \langle \mathrm{D}^2 F(x-ty+2s y)y, y\rangle s\mathrm{d}s\mathrm{d}t\\
\lesssim& \| D^{2}F \|_{L^{\infty}(B_{3/2} \setminus B_{1/2})} F(y)^{2}.
\end{align*}
The last inequality holds due to equivalence of norms on $\R^m\otimes \R^n$.

On the other hand, when $F(y) \ge 1/2$,
$$
F(x+y) + F(x-y) - 2 \le F(x) + F(y) + F(x)+ F(y) - 2 = 2 F(y) \le 4 F(y)^{2}.
$$
In particular, for all $\epsilon \in [0, \infty)$, $\tau_{F}(\epsilon) \lesssim (1 + \|D^{2} F \|_{L^{\infty}(B_{3/2} \setminus B_{1/2})}) \epsilon^{2}$ verifying $F$ is $2$-smooth.

Now suppose $\rho \in C^{1}(\R^{m} \otimes \R^{n} \setminus \{0\}, \R)$, $F = \rho^{2}$, and \eqref{e:unifelliptic} holds. Then, the $2$-convexity of $\rho$ follows from Theorem \ref{t:xuroach}(3). Indeed, \eqref{e:unifelliptic} is precisely \eqref{e:sigmaconvex} when $\gamma=\sigma=2$.
\end{proof}

We next study the convexity and smoothness properties of the spaces $(\X \otimes \R^{n}, F)$, when $F$ is defined via \eqref{e:defF}.

\begin{proposition} \label{p:norms}
Let $(\X, \rho)$ be a Banach space and define $F : \X \otimes \R^{n} \to \R$ by \eqref{e:defF}.

Then
\begin{enumerate}
\item[(i)] $(\X \otimes \R^{n}, F)$ is a Banach space and $(\X^{\ast} \otimes \R^{n}, F_{\ast})$ is its dual. Here we have defined ${F_{\ast}(x^{\ast}) = \left( \sum_{i} \rho_{\ast}(x^{\ast}_{i})^{p^\prime} \right)^{1/p^\prime}}$. 
\item[(ii)] $(\X \otimes \R^{n}, F)$ is reflexive if and only if $(\X,\rho)$ is reflexive.
\item[(iii)] If $\rho$ is $\tau$-smooth then $F$ is $\min(p,\tau$)-smooth.
\item[(iv)] If $\rho$ is $\sigma$-convex then $F$ is $\max(p,\sigma$)-convex.
\item[(v)] $F(x \otimes \xi) = \rho(x) \| \xi \|_{\ell^p}$.
\end{enumerate}
\end{proposition}

\begin{proof}
(i) follows from the linearity of the tensor product and (ii) follows from (i). 

We next check (iii). Let $x=(x_1,\ldots,x_n)\in \X\otimes \R^n$ and $y=(y_1,\ldots,y_n)\in \X\otimes \R^n$. Note by direct calculation that
\begin{align*}
j_F^p(x) = (j_\rho^p(x_1),\ldots,j_\rho^p(x_n)).
\end{align*}
Hence, by (i),
\begin{align*}
F_{\ast} \left( j_{F}^{p}(x) - j_{F}^{p}(y) \right) = \left(\sum_{i} \rho_{\ast} \left( j_{\rho}^{p}(x_{i}) - j_{\rho}^{p}(y_{i}) \right)^{p^{\prime}} \right)^{\frac{1}{p^{\prime}}}.
\end{align*}
Since $\rho$ is $\tau$-smooth, Theorem \ref{t:xuroach} then implies,
\begin{align} \label{e:11}
F_{\ast} \left( j_{F}^{p}(x) - j_{F}^{p}(y) \right) \le \left( \sum_{i} \left( \left( \rho(x_{i}) + \rho(y_{i}) \right)^{p-\tau} \rho(x_{i} -y_{i})^{\tau-1}\right)^{p^{\prime}} \right)^{\frac{1}{p^{\prime}}}.
\end{align}
When $p \le \tau$, the triangle inequality says $(\rho(x_{i}) + \rho(y_{i}))^{p-\tau} \le \rho(x_{i}-y_{i})^{p-\tau}$. So \eqref{e:11} implies
$$
F_{\ast} \left(j_{F}^{p}(x) - j_{F}^{p}(y) \right) \le \left( \sum_{i} \rho(x_{i}-y_{i})^{p} \right)^{\frac{p-1}{p}} = F(x-y)^{p-1}.
$$
By Theorem \ref{t:xuroach}(c), this implies $F$ is $p$-smooth when $p = \min (p,\tau)$.  On the other hand, when $\tau \le p$, it follows from \eqref{e:11}, monotonicity of $t \mapsto t^{p-\tau}$, and Jensen's inequality
\begin{align*}
F_{\ast} \left( j_{F}^{p}(x) - j_{F}^{p}(y) \right)  & \le \left( F(x) + F(y) \right)^{p-\tau} \left( \sum_{i} \rho(x_{i} - y_{i})^{p \left(\frac{\tau-1}{p-1} \right)} \right)^{\frac{p-1}{p}} \\
& \lesssim \left( F(x) + F(y) \right)^{p-\tau} \left( \sum_{i} \rho(x_{i}-y_{i})^{p} \right)^{\frac{\tau-1}{p}}  \\
& = (F(x) + F(y))^{p-\tau} F(x-y)^{\tau-1}.
\end{align*}
Now, Theorem \ref{t:xuroach}(c) verifies $\tau$-smoothness of $F$ when $\tau = \min(p,\tau)$.  This verifies (iii). 

To prove (iv), note that by Proposition \ref{p:smoothconvexduality}, $\rho$ is $\sigma$-convex if and only if $\rho_{\ast}$ is $\sigma^{\prime}$-smooth. Thus by part (i) and (iii) of this Proposition, $F_{\ast}$ is $\min(p^\prime,\sigma^{\prime})$-smooth. Since $p^{\prime} \le \sigma^{\prime} \iff \sigma \le p$, another application of Proposition \ref{p:smoothconvexduality} implies $(F_{\ast})_{\ast} = F$ is $\max(p,\sigma)$-convex as desired.

Finally (v) is a straightforward computation, since $x \otimes \xi = ( \xi_{1} x,  \dots, \xi_{n} x)$.
\end{proof}

\section{A \texorpdfstring{$V$}{}-functional estimate in Banach spaces}\label{s:VFunc}
In this section, we prove Lemma \ref{lem:VFunc}. In the case when $\X=\R^m$ and $\rho$ is the Euclidean norm this result is well-known. However, even in this set-up our elementary proof is new. For the convenience of the reader, we restate the lemma here. Recall the definition of $V_{p,q}$ from \eqref{e:vfuncdef}.

\begin{lemma} \label{l:VFuncEquiv}
Let $(\Y, F)$ be a Banach space and fix $p,q > 1$. Then, for any $\xi,\,\eta\in \Y$,
\begin{equation} \label{e:VFuncEquiv}
F \left( V_{p,q}(\xi) - V_{p,q}(\eta) \right) \sim \left( F(\xi) + F(\eta) \right)^{\frac{p-q}{q}} F(\xi - \eta).
\end{equation}
\end{lemma}

\begin{proof}
{\bf Case 1:} $p-q \geq 0$.

For $\xi, \eta \in \Y$, convexity of $F$ implies the difference quotient
$$
t \mapsto \frac{ F(\xi + t \eta) - F(\xi)}{t}
$$
is a non-decreasing function for $t \in [0, \infty)$. Suppose without loss of generality, $F(\xi) \ge F(\eta)$ and define $\kappa = \frac{F(\eta)}{F(\xi)} \le 1$. Since $p-q \geq 0$, $\kappa^{\frac{p-q}{q}} \le 1$. Using $F(\xi) \kappa = F(\eta)$ and convexity of $F$, it follows
\begin{align*}
F \left( V_{p,q}(\xi) - V_{p,q}(\eta) \right) & = F(\xi)^{\frac{p-q}{q}} \kappa^{\frac{p-q}{q}} \frac{F\left(\xi - \kappa^{\frac{p-q}{q}} \eta \right) - F(\xi)}{\kappa^{\frac{p-q}{q}}} + F(\xi)^{\frac{p}{q}} \\
& \le F(\eta)^{\frac{p-q}{q}} \left[ F(\xi - \eta) - F(\xi) \right] + F(\xi)^{\frac{p}{q}} \\
& = F(\eta)^{\frac{p-q}{q}} F(\xi - \eta) + \left( F(\xi)^{\frac{p-q}{q}} - F(\eta)^{\frac{p-q}{q}} \right) F(\xi) \\
& \lesssim F(\eta)^{\frac{p-q}{q}}F(\xi -\eta) + F(\xi)^{\frac{p-q}{q}}F(\xi - \eta),
\end{align*}
where the last line follows by the mean-value theorem applied to $t \mapsto t^{\frac{p-q}{q}}$ on $[F(\eta), F(\xi)]$. This verifies 
$$
F(V_{p,q}(\xi) - V_{p,q}(\eta)) \le \left( F(\xi)^{\frac{p-q}{q}} + F(\eta)^{\frac{p-q}{q}} \right) F(\xi - \eta)
$$ 
which suffices since for $a,b > 0$ and any $\alpha > 0$, $a^{\alpha} + b^{\alpha} \le 2(a+b)^{\alpha}$.
To confirm the reverse inequality, assume again without loss of generality, that $F(\xi) \geq F(\eta)$.
Then, using the reverse triangle inequality,
\begin{align*}
F\left(V_{p,q}(\xi) - V_{p,q}(\eta) \right) & = F \left( F(\xi)^{\frac{p-q}{q}} (\xi - \eta) - \left( F(\eta)^{\frac{p-q}{q}} - F(\xi)^{\frac{p-q}{q}} \right) \eta \right) \\
& \ge F(\eta)^{\frac{p-q}{q}} F(\xi - \eta) - \left( F(\eta)^{\frac{p-q}{q}} - F(\xi)^{\frac{p-q}{q}} \right) F(\eta) \\
& \ge F(\eta)^{\frac{p-q}{q}} F(\xi-\eta) \\
& \gtrsim (F(\eta) + F(\xi))^{\frac{p-q}{q}} F(\xi - \eta).
\end{align*}
This is the desired inequality.

{\bf Case 2:} $p-q < 0$. 
Suppose $F(\xi) \ge F(\eta)$ and $F(\xi) \kappa = F(\eta)$. Then, $\kappa^{\frac{q-p}{q}} \le 1$. Akin to the previous case, we estimate,
\begin{align*}
F \left( V_{p,q}(\xi) - V_{p,q}(\eta) \right) &= F(\eta)^{\frac{p-q}{q}} \kappa^{\frac{q-p}{q}} \frac{ F( \kappa^{\frac{q-p}{q}} \xi - \eta) - F(\eta) }{\kappa^{\frac{q-p}{q}}} + F(\eta)^{\frac{p}{q}} \\
& \le F(\xi)^{\frac{p-q}{q}} \left[ F(\xi - \eta) - F(\eta) \right] + F(\eta)^{\frac{p}{q}} \\
& = F(\xi)^{\frac{p-q}{q}} F(\xi-\eta) + F(\eta)^{\frac{p}{q}} -F(\eta) F(\xi)^{\frac{p-q}{q}}.
\end{align*}
Since $F(\xi) \ge F(\eta)$,
$$
F(\eta)^{\frac{p}{q}} - F(\eta) F(\xi)^{\frac{p-q}{q}} \le F(\xi)^{\frac{p-q}{q}} \lvert F(\xi) - F(\eta) \rvert \le F(\xi)^{\frac{p-q}{q}} F(\xi-\eta).
$$
Recalling $p-q < 0$, combining the two previous estimates with the observation that when $\lvert a\rvert \ge \lvert b\rvert$ and $\alpha < 0$, $\lvert a\rvert^{\alpha}\le 2^{-\alpha} \lvert a+b\rvert^{\alpha}$ finishes the first inequality. We now show the opposite. Assume without loss of generality $F(\xi) > F(\eta) > 0$, noting that if $F(\eta)=0$ or $F(\xi)=F(\eta)$, there is nothing to prove. Then, using the reverse triangle inequality,
\begin{align*}
 F(V_{p,q}(\eta)-V_{p,q}(\xi)) 
& = F \left( F(\eta)^{\frac{p-q}{q}}(\eta - \xi) - \left(F(\xi)^{\frac{p-q}{q}} - F(\eta)^{\frac{p-q}{q}} \right) \xi \right) \\
& \ge F(\eta)^{\frac{p-q}{q}}F(\eta-\xi) - \left( F(\xi)^{\frac{p-q}{q}} - F(\eta)^{\frac{p-q}{q}} \right) F(\xi) \\
& \ge F(\eta)^{\frac{p-q}{q}} F(\eta- \xi) \\
& \gtrsim \left(F(\eta) + F(\xi) \right)^{\frac{p-q}{q}} F(\eta-\xi).
\end{align*}
Since $F( \zeta) = F(- \zeta)$ for all $\zeta \in \Y$, this is the desired inequality.
\end{proof}

We explicitly state a consequence of Lemma \ref{l:VFuncEquiv} in our set-up.
\begin{corollary} \label{c:VFuncEquiv} Let $(\Y,F)$ be a Banach space. 
Suppose $u \in \Ww^{1,p}_{\loc}(\Omega,\Y)$. Then for all $p,q\geq 0$,
\begin{align} \label{e:0}
F& \left( \Delta_{h} V_{p,q} (\nabla u) \right) \sim (F(\nabla u) + F(\nabla u_{h}))^{\frac{p-q}{q}} F(\Delta_{h} \nabla u) .
\end{align}

In particular, if $p \ge q$,
\begin{equation} \label{e:1}
F(\Delta_{h} \nabla u)^{p} \lesssim F \left( \Delta_{h} V_{p,q}(\nabla u)  \right)^{q},
\end{equation}
and when $p \le q$, if $\eta \in C_{c}(\Omega)$ and $A = \spt \eta$, $\beta \geq 0$,
\begin{equation} \label{e:2.0}
\int \eta^{\beta}F(\Delta_{h} \nabla u)^{p} \lesssim \left( \int \eta^{\frac{q \beta}{p}} F\left( \Delta_{h} V_{p,q}(\nabla u) \right)^{q} \right)^{\frac{p}{q}} \left( \int_{A} (F(\nabla u) + F(\nabla u_{h}))^{p}\right)^{1 - \frac{p}{q}}
\end{equation}
\end{corollary}

\begin{proof}
Applying Lemma \ref{l:VFuncEquiv} with $\xi= \nabla u$ and $\eta = \nabla u_{h}$ confirms \eqref{e:0}. 

Suppose $p \ge q$. Then $F(\Delta_{h} \nabla u)^{p-q} \le (F(\nabla u) + F(\nabla u_{h}))^{p-q}$. So \eqref{e:0} implies
$$
F \left( \Delta_{h} V_{p,q}(\nabla u) \right)^{q} \sim (F(\nabla u) + F(\nabla u_{h}))^{p-q} F(\Delta_{h} \nabla u)^{q} \ge F(\Delta_{h}u)^{p},
$$
confirming \eqref{e:1}. 

Suppose $p \le q$. $\eta \in C_{c}(\Omega)$ and $A = \spt \eta$. Choosing $\tau= \frac{p-q}{q}$, \eqref{e:0} implies
$$
F(\Delta_{h} \nabla u)^{p} \sim F \left( \Delta_{h}V_{p,q}(\nabla u) \right)^{p} \left( F(\nabla u) + F(\nabla u_{h}) \right)^{\frac{p(q-p)}{q}}.
$$ 
Now \eqref{e:2.0} follows by applying Hölder's inequality with exponents $\frac{q}{p}$ and $\frac{q}{q-p}$ to
\begin{align*}
\int \eta^{\beta} F( \Delta_{h} \nabla u)^{p} & \lesssim \int \eta^{\beta} F\left(\Delta_{h}  V_{p,q}(\nabla u)  \right)^{p} (F(\nabla u) + F(\nabla u_{h}))^{\frac{p(q-p)}{q}}.
\end{align*}
\end{proof}

\section{Proof of main theorem}\label{s:proof}
We restate the main theorem for the readers convenience.
\begin{theorem} \label{t:main}
Fix $1 < \gamma < \infty, \tau \in (1,2]$, and $\sigma \in [2,\infty)$. Let $(\X, \rho)$ be a Banach space, $\Omega \subset \R^{n}$ a domain, and $F$ a $\sigma$-convex and $\tau$-smooth norm on $\X \otimes \R^{n}$ satisfying for some norm $\| \cdot \|$ on $\R^{n}$
\begin{align}\label{e:assumptionF}
F(\xi_{1}, \dots, \xi_{n} )\lesssim\|( \rho(\xi_{1}),\dots, \rho(\xi_{n})\|,
\end{align}
Set $\alpha = \max ( \min (\gamma^{\prime}, \sigma^{\prime} ) , \min ( \gamma, \tau))$, and $\alpha_{\ast} = \alpha -1$. Assume $f\in \Bb^{\alpha_\ast,\gamma^\prime}_{\infty}(\Omega)$. Suppose $u \in \Ww^{1,\gamma}(\Omega;\X)$ solves \eqref{e:fullgen}. Then
\begin{equation} \label{e:VFuncReg}
V(\Der u) \in \Bb^{\frac{\alpha}{\sigma},\sigma}_{\infty,\loc}(\Omega) \quad \text{and} \quad V^{\ast}(j_F^\gamma(\Der u)) \in \Bb^{\frac{\alpha}{\tau^{\prime}}, \tau^{\prime}}_{\infty,\loc}(\Omega)
\end{equation}
In fact, for any $B_{r} \Subset B_{s} \Subset \Omega$,
\begin{align}\label{e:VFuncNorm}
\|V(\Der u)\|^{\sigma}_{\Bb^{\frac{\alpha}{\sigma},\sigma}_{\infty}(B_{r})}+\|V_{\ast}(j_F^\gamma(\Der u)\|^{\tau^{\prime}}_{\Bb^{\frac{\alpha}{\tau^{\prime}},\tau^{\prime}}_{\infty}(B_{r})} \lesssim_{(s-r)}  \|u\|_{\Ww^{1,\gamma}(B_{s})}^\gamma + \|f\|_{\Bb_{\infty}^{\alpha_{\ast},\gamma^\prime}(B_{s})}^{\gamma^\prime} .
\end{align}
Moreover, 
\begin{equation} \label{e:FuncReg}
\nabla u \in \Bb_{\infty,\loc}^{\frac{\alpha}{\max\{\sigma,\gamma\}},\gamma}(\Omega) \quad \text{and} \quad j_{F}^{\gamma}(\nabla u) \in \Bb_{\infty,\loc}^{\frac{\alpha}{\max\{\tau^{\prime}, \gamma^{\prime}\}},\gamma^{\prime}}(\Omega).
\end{equation}
\end{theorem}

A key tool in the proof is the following Lemma: 

\begin{lemma}\label{lem:deriv}
Let $1\leq p<\infty$.
Suppose $f\in \LL^p_{loc}(\Omega)$. Then for all $\phi\in \WW^{1,p^\prime}_0(\Omega)$ and $h = \lvert h\rvert e$ with $\lvert e\rvert=1$,
\begin{align*}
\int_\Omega \phi\Delta_h f = -\lvert h\rvert \int_\Omega \frac{\p \phi}{\phi e}\left(\int_0^1 f(x+t h )\mathrm{d} t\right).
\end{align*}
\end{lemma}
\begin{proof}
The identity holds for smooth functions, since
$$
\Delta_h f = \frac{\p}{\p e}\int_0^1 f(x+th)\mathrm{d} t.
$$
The result follows by approximation.
\end{proof}

\begin{proof}[Proof of Theorem \ref{t:main}]
Due to equivalence of norm on $\R^n$, we deduce from \eqref{e:assumptionF},
\begin{equation} \label{e:fstruct}
F( z \otimes \xi ) \lesssim \rho(z) \lb \xi\rb \quad \forall z \in \X ~ \forall \xi \in \R^{n}.
\end{equation}
In fact, due to equivalence of norms on $\R^n$,
\begin{equation}
F(z\otimes \xi) \lesssim \|\xi_1\rho(z),\ldots,\xi_n\rho(z)\|\lesssim \rho(z)\lvert \xi \rvert.
\end{equation}

Suppose for some $r<s$, $B_{r}\Subset B_{s}\Subset\Omega$. Let $\eta\in C_c^\infty(\Omega; \R)$ be a smooth non-negative function supported on $r+\frac{s-r} 3$ with $\lvert \Der \eta \rvert \lesssim  (s-r)^{-1}, \lvert \Der^{2} \eta\rvert \lesssim (s-r)^{-2}$ such that $\eta =1 $ on $B_r$. Let $h= \lvert h\rvert e\in \R^n$ with $\lvert e\rvert=1$ and write $u_h(x) = u(x+h)$ for $\lvert h\rvert<\frac{s-r} 3$.

Recall that we denote $\Delta_h u(x) = u(x)-u(x+h)$. Let $\beta>0$ to be determined at a later stage and test \eqref{e:fullgen} against $\Delta_{-h} \left(\eta^\beta\Delta_h u\right)$. Using discrete integration by parts, we find 
\begin{align}\label{e:EL} 
I \defeq \int_{\Omega} \langle \Delta_h j_F^\gamma(\Der u),\Der(\eta^\beta \Delta_h u)\rangle_{\X \otimes \R^{n}} =\int_{\Omega} \eta^{\alpha} \Langle\Delta_h u, \Delta_{h}f\Rangle_\X \eqdef II.
\end{align}
We write
\begin{align*}
I = \int_{\Omega} \langle \Delta_h j_F^\gamma(\Der u), \eta^{\beta} \Delta_h \Der u\rangle_{\X\otimes \R^n}+\beta\int_{\Omega} \eta^{\beta-1}\langle \Delta_h j_F^\gamma(\Der u),\Delta_h u \otimes \Der \eta \rangle \eqdef A_1 + A_2.
\end{align*}

Note that due to \eqref{e:ident},
\begin{align*}
\langle \Delta_h j_F^\gamma(\Der u), \Delta_h \Der u\rangle_{\X\otimes \R^n} = \langle\Delta_h j_F^\gamma(\Der u),\Delta_h j^{\gamma'}_{F_\ast}(j^\gamma_F(\Der u))\rangle_{\X \otimes \R^{n}}.
\end{align*}
Therefore, we bound $A_{1}$ below by \eqref{e:sigmaconvex} and \eqref{e:tauprimeconvex}, as
\begin{align}  \label{e:A1}
\nonumber A_1\gtrsim &\int_{\Omega} \eta^{\beta} \left( F(\Der u) + F(\Der u_{h}) \right)^{\gamma-\sigma} F(\Delta_h \Der u)^{\sigma}  \\
&\quad+\int_{\Omega} \eta^{\beta} \left( F_{\ast}(j_{F}^{\gamma}(\Der u))+ F_{\ast}(j_{F}^{\gamma}(\Der u_{h})) \right)^{\gamma^{\prime}-\tau^{\prime}} F_{\ast}(\Delta_h j_F^\gamma(\Der u))^{\tau^{\prime}}.
\end{align}
It follows from \eqref{e:0} that \eqref{e:A1} is equivalent to
\begin{equation} \label{e:A1vfunc}
A_{1} \gtrsim \int_{\Omega} \eta^{\beta} F\left( \Delta_{h} V(\Der u) \right)^{\sigma} + \int_{\Omega} \eta^{\beta} F_{\ast} \left( \Delta_{h} V^{\ast}(j_F^\gamma(\Der u)) \right)^{\tau^{\prime}}.
\end{equation}
On the other hand, consecutively using Fenchel, Cauchy-Schwarz, \eqref{e:dqchar} and Young's, we obtain
\begin{align}\label{e:IIvfunc}
\nonumber \lvert II\rvert & \lesssim\int_{\Omega} \eta^{\beta} \rho(\Delta_h u) \rho_{\ast}(\Delta_h f)\lesssim \|\rho(\Delta_h u)\|_{\Ll^\gamma(B_{r+1/3(s-r)})}\|\rho_\ast(\Delta_h f)\|_{\Ll^{\gamma^\prime}(B_{r+1/3(s-r)})}\\
& \lesssim \lvert h\rvert^{1 + \alpha_{\ast}} \left(\|u\|_{\Ww^{1,\gamma}(B_{s})}^\gamma+ \|f\|_{\Bb^{\alpha_{\ast},\gamma^{\prime}}_{\infty}(B_{s})}^{\gamma^\prime}\right).
\end{align}

The claimed regularity statement will follow from combining the observations made so far with various estimates on $A_2$: a primal estimate, carried out separately in the regime $\gamma\leq\sigma$ and $\gamma\geq \sigma$ and similar dual estimates in the regime $\gamma^\prime \leq \tau^\prime$ and $\gamma^\prime \geq \tau^\prime$.

{\bf Case 1.1: A primal estimate if $\gamma \le \sigma$.}

Applying Lemma \ref{lem:deriv} with $f = j_{F}^{\gamma}(\nabla u)$ and $\phi = \eta^{\beta-1} \Delta_{h} u \otimes \nabla \eta$, we get
\begin{align*}
\lvert A_{2}\rvert &\lesssim \lvert h\rvert \int_{\Omega} \left \langle \frac{\partial}{\partial e} \left( \eta^{\beta-1} \Delta_{h} u \otimes \nabla \eta \right), \int_{0}^{1} j_{F}^{\gamma}(\nabla u(x+th)) \mathrm{d}t \right \rangle \\
& = \lvert h\rvert \int_{\Omega} \left \langle \Delta_{h}u \otimes \frac{\partial}{\partial e} (\eta^{\beta-1} \nabla \eta) ,  \int_{0}^{1} j_{F}^{\gamma}(\nabla u (x+th)) \mathrm{d}t \right \rangle \\
& \quad + \lvert h\rvert \int_{\Omega}  \left \langle \frac{\partial }{\partial e} (\Delta_{h} u) \otimes (\eta^{\beta-1} \nabla \eta) , \int_{0}^{1} j_{F}^{\gamma}(\nabla u (x+th)) \mathrm{d}t \right \rangle = A_{21} + A_{22} 
\end{align*}
Using $F_{\ast}(j_{F}^{\gamma}(\xi)) = F(\xi)^{\gamma-1}$, \eqref{e:fstruct}, and \eqref{e:dqchar}, implies
\begin{align}
\nonumber \lvert A_{21}\rvert & \lesssim \lvert h\rvert \int_{\Omega} F \left( \Delta_{h} u  \otimes \frac{\partial}{\partial e} \left(\eta^{\beta-1} \nabla \eta \right) \right) F_{\ast} \left( \int_{0}^{1} j_{F}^{\gamma}(\nabla u(x+th) ) \mathrm{d}t \right) \\
\nonumber & \le \lvert h\rvert \int_{\Omega} \left\{ \rho \left( \Delta_{h} u \right) \lvert \frac{\partial}{\partial e} (\eta^{\alpha-1} \nabla \eta) \rvert \int_{0}^{1} F_{\ast} \left( j_{F}^{\gamma}(\nabla u(x+th))\mathrm{d}t \right) \right\} \\
\nonumber & \lesssim \lvert h\rvert \|u\|_{\Ww^{1,\gamma}(B_{s})} \lvert h\rvert \left(\int_{B_{r+ \frac{2}{3}(s-r)}} \int_{0}^{1} F(\nabla u)^{\gamma}  \mathrm{d} t \right)^{\frac{1}{\gamma^{\prime}}} \\
\label{e:a21max}  & \lesssim \lvert h\rvert^{2} \|u\|_{\Ww^{1,\gamma}(B_{s})}  \left( \int_{B_{s}} F(\nabla u)^{\gamma}\right)^{\frac{1}{\gamma^{\prime}}} 
\lesssim\lvert h\rvert^{2} \|u\|_{\Ww^{1,\gamma}(B_{s})}^{\gamma}.
\end{align}

Choosing $\beta= \sigma^{\prime}$ and arguing similarly, using \eqref{e:fstruct} and \eqref{e:assumptionF}
\begin{align} \label{e:a22case0}
\nonumber \lvert A_{22}\rvert & \lesssim \lvert h\rvert \int_{\Omega} F \left( \frac{\partial}{\partial e}(\Delta_{h} u ) \otimes(\eta^{\sigma^{\prime}-1} \nabla \eta) \right) \int_{0}^{1} F_{\ast} \left( j_{F}^{\gamma}(\nabla u(x+th)) \right)\mathrm{d}t \\
\nonumber & \lesssim \lvert h\rvert \int_{\Omega} \rho \left( \Delta_{h} \frac{\partial u}{\partial e} \right) \lvert \eta^{\sigma^{\prime}-1} \nabla \eta\rvert \int_{0}^{1} F(\nabla u(x+th))^{\gamma-1} \mathrm{d}t \\
\nonumber & \lesssim \lvert h\rvert \int_{\Omega} F( \Delta_{h} \Der u ) \eta^{\sigma^{\prime}-1} \int_{0}^{1} F(\nabla u(x+th))^{\gamma-1}\mathrm{d}t \\
& \lesssim \lvert h\rvert \left(  \int_{\Omega} F(\Delta_{h} \nabla u)^{\gamma} \eta^{\gamma(\sigma^{\prime}-1)} \right)^{\frac{1}{\gamma}} \left( \int_{B_{\rho_{2}}} \int_{0}^{1} F(\nabla u(x+th))^{\gamma} \right)^{\frac{1}{\gamma^{\prime}}}.
\end{align}

Using \eqref{e:2.0} with $p = \gamma \le \sigma= q$ we find,
$$
\left( \int_{\Omega} F(\Delta_{h} u)^{\gamma} \eta^{\gamma(\sigma^{\prime}-1)} \right)^{\frac{1}{\gamma}} \lesssim \left( \int \eta^{\sigma^{\prime}} F \left( \Delta_{h} V \right)^{\sigma} \right)^{\frac{1}{\sigma}} \left( \int_{B_{\rho_{2}}} (F(\nabla u) + F(\nabla u_{h}))^{\gamma} \right)^{\frac{\sigma-\gamma}{\sigma \gamma}}.
$$

Plugging this into \eqref{e:a22case0} and applying Young's inequality yields for any $\e>0$,
\begin{align} \label{e:a22max}
\nonumber \lvert A_{22}\rvert & \lesssim \epsilon \int \eta^{\sigma^{\prime}} F \left( \Delta_{h}V \right)^{\sigma}+ C(\epsilon) \lvert h\rvert^{\sigma^{\prime}} \|\nabla u\|_{\Ll^{\gamma}(B_{s})}^{\sigma^{\prime} (\gamma-1)} \| \nabla u \|_{\Ll^{\gamma}(B_{s})}^{\frac{\sigma-\gamma}{\gamma(\sigma-1)}} \\
& = \epsilon \int \eta^{\sigma^{\prime}} F \left( \Delta_{h}V  \right)^{\sigma} + C(\epsilon) \lvert h\rvert^{\sigma^{\prime}} \|\nabla u\|_{\Ll^{\gamma}(B_{s})}^{\gamma}
\end{align}

From the PDE, we have $A_{1} + A_{2} = I \le \lvert II\rvert$. In particular $A_{1} - \lvert A_{22}\rvert  \le \lvert A_{21}\rvert  + \lvert II\rvert$. Estimating each term respectively by \eqref{e:A1}, \eqref{e:a22max}, \eqref{e:a21max}, and \eqref{e:IIvfunc} yields, choosing $\e$ sufficiently small,
\begin{align}\label{e:final1}
& \int_{\Omega} \eta^{\beta} F\left( \Delta_{h} V \right)^{\sigma} + \int_{\Omega} \eta^{\beta} F_{\ast} \left( \Delta_{h} V^{\ast}\right)^{\tau^{\prime}} \nonumber\\
 \lesssim_{\epsilon}& \left(\lvert h\rvert^{\sigma^{\prime}}+\lvert h\rvert^2+\lvert h\rvert^{1+\alpha_\ast}\right) \left(\|u\|_{\WW^{1,\gamma}(B_s)}^\gamma+\|f\|_{\Bb^{\alpha_\ast,\gamma^\prime}_\infty(B_s)}^{\gamma^\prime}\right).
\end{align}

Noting $\alpha_\ast = \sigma^\prime$ and dividing by $\lvert h\rvert^{\sigma^{\prime}}$ implies
$$
  \| V (\Der u)  \|_{ \Bb^{\sigma^{\prime}-1,\sigma}_{\infty}(B_{r})} +  \| V^{\ast}(j_F^\gamma(\Der u)) \|_{\Bb^{\frac{\sigma^{\prime}}{\tau^{\prime}},\tau^{\prime}}_{\infty}(B_{r})}\lesssim_{(s-r)} \|u\|_{\Ww^{1,\gamma}(B_{s})}^\gamma +\|f\|_{\Bb^{\sigma^{\prime}-1,\gamma^{\prime}}_{\infty}(B_{s})}^{\gamma^\prime}.
$$
\textbf{Case 1.2: A primal estimate if $\gamma\geq \sigma$}

We return to \eqref{e:a22case0}, choose $\beta =\gamma \ge \sigma \ge 2$ and estimate using H\"older's inequality and Young's inequality,
\begin{align*}
\nonumber \lvert A_{22}\rvert \lesssim& \lvert h\rvert (s-r)^{-1}\left(\int_\Omega  F(\Delta_h \Der u)^\gamma\eta^\gamma\right)^\frac 1 \gamma\left(\int_{B_{r+\frac{s-r} 3}}\eta^{(\gamma-2)\gamma^{\prime}}\int_0^1 F(\Der u(x+th))^\gamma dt\right)^\frac 1 {\gamma^\prime}\\
\lesssim& \e \int_\Omega \eta^\gamma F(\Delta_h \Der u)^{\gamma}+C(\e)\lvert h\rvert^{\gamma^\prime}\int_{B_s} F(\Der u)^\gamma.
\end{align*}
Applying \eqref{e:1} yields
\begin{equation} \label{e:A22c12}
\lvert A_{22}\rvert \lesssim \epsilon \int_{\Omega} F \left( \Delta_{h} V \right)^{\sigma}  + C(\epsilon) \lvert h\rvert^{\gamma^{\prime}} \int_{B_{s}} F(\nabla u)^{\gamma}.
\end{equation} 
Now estimating each term in $A_{1} - \lvert A_{22}\rvert \le \lvert A_{21}\rvert + \lvert II\rvert$, respectively by \eqref{e:A1}, \eqref{e:A22c12}, \eqref{e:a21max}, and \eqref{e:IIvfunc}, yields
\begin{align}\label{e:final2}
& \int_{\Omega} \eta^{\beta} F\left( \Delta_{h} V \right)^{\sigma} + \int_{\Omega} \eta^{\beta} F_{\ast} \left( \Delta_{h} V^{\ast} \right)^{\tau^{\prime}} \nonumber\\
\lesssim_{\epsilon}& \left(\lvert h\rvert^{\gamma^\prime}+\lvert h\rvert^2+\lvert h\rvert^{1+\alpha_\ast}\right) \left(\|u\|_{\Ww^{1,\gamma}(B_{s})}^\gamma+ \|f\|_{B^{\alpha_{\ast},\gamma^{\prime}}_{\infty}(B_{s})}^{\gamma^\prime}\right).
\end{align}

Noting that $\alpha_\ast + 1 = \gamma^\prime$, dividing by $\lvert h\rvert^{\gamma^{\prime}}$ yields

$$
\| V(\Der u)  \|_{ \Bb^{\frac{\gamma^{\prime}}{\sigma},\sigma}_{\infty}(B_{r})} +  \| V^{\ast}(j_F^\gamma(\Der u))\|_{\Bb^{\frac{\gamma^{\prime}}{\tau^{\prime}},\tau^{\prime}}_{\infty}(B_{r})} \lesssim_{(s-r)} \|u\|_{W^{1,\gamma}(B_{s})}^\gamma+ \|f\|_{B^{\gamma^{\prime}-1,\gamma^{\prime}}_{\infty}(B_{s})}^{\gamma^\prime}.
$$

{\bf Case 2.1.: A dual estimate if $\tau^{\prime} \ge \gamma^{\prime}$.} 

Choose $\beta = \tau$. From Fenchel, Hölder, and \eqref{e:dqchar} it follows
\begin{align} \label{e:c21a2.1}
\nonumber \lvert A_{2}\rvert &= \tau \lvert \int_{\Omega} \eta^{\tau-1}\langle \Delta_h j_F^\gamma(\Der u),\Delta_h u \otimes \Der \eta \rangle \rvert \\
\nonumber & \lesssim \int_{\Omega} \eta^{\tau-1} F(\Delta_{h} u \otimes \nabla \eta) F_{\ast} \left( \Delta_{h} j_{F}^{\gamma}(\nabla u) \right) \\
\nonumber & \lesssim \left( \int_{B_{\rho_{1}}} \rho(\Delta_{h} u)^{\gamma} \lvert\nabla \eta\rvert^{\gamma} \right)^{\frac{1}{\gamma}} \left( \int \eta^{\gamma^{\prime}(\tau-1)} F_{\ast} \left(\Delta_{h} j_{F}^{\gamma}(\nabla u) \right)^{\gamma^{\prime}} \right)^{\frac{1}{\gamma^{\prime}}}  \\
& \le \lvert h\rvert \|u\|_{\Ww^{1,\gamma}(B_{s})} \left( \int \eta^{\gamma^{\prime}(\tau-1)} F_{\ast} \left(\Delta_{h} j_{F}^{\gamma}(\nabla u) \right)^{\gamma^{\prime}} \right)^{\frac{1}{\gamma^{\prime}}}  
\end{align}
Since $V^{\ast}(j_F^\gamma(\Der u)) = V^{\ast}_{\gamma^{\prime},\tau^{\prime}}(j_{F}^{\gamma}(\nabla u))$, up to replacing $F$ with $F_{\ast}$ and $\nabla u$ with $j_{F}^{\gamma}(\nabla u)$, it follows from \eqref{e:2.0} with $p = \gamma^{\prime}$ and $q = \tau^{\prime}$, that
\begin{align} \label{e:c21a2.2}
&\nonumber  \bigg( \int \eta^{\gamma^{\prime}(\tau-1)} F_{\ast} \left(\Delta_{h} j_{F}^{\gamma}(\nabla u) \bigg)^{\gamma^{\prime}} \right)^{\frac{1}{\gamma^{\prime}}} \\
 \le& \left( \int \eta^{\tau} F_{\ast} \left( \Delta_{h} V^{\ast}(j_F^\gamma(\Der u)) \right)^{\tau^{\prime}} \right)^{\frac{1}{\tau^{\prime}}}  \left( \int_{B_{\rho_{2}}} \left( F_{\ast} ( j_{F}^{\gamma}(\nabla u)) + F_{\ast}(j_{F}^{\gamma}(\nabla u)) \right)^{\gamma^{\prime}} \right)^{\frac{\tau^{\prime}-\gamma^{\prime}}{\tau^{\prime} \gamma^{\prime}}}  \nonumber\\
 \lesssim& \left( \int \eta^{\tau}F_{\ast} \left( \Delta_{h} V^\ast(j_F^\gamma(\Der u))\right)^{\tau^{\prime}} \right)^{\frac{1}{\tau^{\prime}}}  \|j_{F}^{\gamma}(\nabla u) \|_{\LL^{\gamma^{\prime}}(B_{s})}^{1 - \frac{ \gamma^{\prime}}{\tau^{\prime}}}.
\end{align}
Recalling $F_{\ast}(j_{F}^{\gamma}(\xi)) = F(\xi)^{\gamma-1}$, we note $\displaystyle 
 \|j_{F}^{\gamma}(\nabla u) \|_{\LL^{\gamma^{\prime}}(B_{s})}^{1 - \frac{ \gamma^{\prime}}{\tau^{\prime}}} = \| \nabla u \|_{\LL^{\gamma}(B_{s})}^{(\gamma-1)(1 - \frac{\gamma^{\prime}}{\tau^{\prime}})} = \| \nabla u\|_{\LL^{\gamma}(B_{s})}^{\frac{\gamma-\tau}{\tau}}$. Then plugging \eqref{e:c21a2.2} into \eqref{e:c21a2.1} and applying Young's inequality yields for any $\e>0$,
\begin{align} \label{e:c21a2}
\nonumber \lvert A_{2}\rvert & \lesssim \epsilon \int \eta^{\tau}F_{\ast} \left( \Delta_{h} V^{\ast}(j_F^\gamma(\Der u)) \right)^{\tau^{\prime}} + C(\epsilon) \lvert h\rvert^{\tau} \|\nabla u\|_{\LL^{\gamma}(B_{s})}^{\gamma-\tau} \|u\|_{\Ww^{1,\gamma}(B_{s})}^{\tau} \\
& \lesssim  \epsilon \int \eta^{\tau}F_{\ast} \left( \Delta_{h} V^{\ast}(j_F^\gamma(\Der u))\right)^{\tau^{\prime}} + C(\epsilon) \lvert h\rvert^{\tau} \|u\|_{\Ww^{1,\gamma}(B_{s})}^{\gamma}
\end{align}

Since the PDE implies $A_{1} \le \lvert A_{2}\rvert + \lvert II\rvert$, choosing $\e$ small enough, and estimating each term consecutively by \eqref{e:A1vfunc},  \eqref{e:c21a2}, and \eqref{e:IIvfunc} yields
\begin{align}\label{e:final3}
& \int_{\Omega} \eta^{\beta} F\left( \Delta_{h} V \right)^{\sigma} + \int_{\Omega} \eta^{\beta} F_{\ast} \left( \Delta_{h} V^{\ast} \right)^{\tau^{\prime}}  \\
\nonumber & \lesssim \left(\lvert h\rvert^{\tau}+\lvert h\rvert^2+\lvert h\rvert^{1+\alpha_\ast}\right)\left( \|u\|_{\Ww^{1,\gamma}(B_{s})}^\gamma+ \|f\|_{B^{\alpha_{\ast},\gamma^\prime}_{\infty}(B_s)}^{\gamma^\prime}\right).
\end{align}

Noting $\alpha_\ast-1 = \tau$, dividing by $\lvert h\rvert^{\tau}$ verifies
$$
 \left \| V(\Der u) \right \|_{ \Bb^{\frac{\tau}{\sigma},\sigma}_{\infty}(B_{r})} + \left \| V^{\ast}(j_F^\gamma(\Der u)) \right\|_{\Bb^{\tau-1,\tau^{\prime}}_{\infty}(B_{r})}  \lesssim_{(s-r)}\|u\|_{\Ww^{1,\gamma}(B_{s})}^\gamma+\|f\|_{B^{\tau-1,\gamma^{\prime}}_{\infty}(B_{s})}^{\gamma^\prime}.
$$

\textbf{Case 2.2.: A dual estimate if $\gamma\leq \tau$.}

We return to estimating $\lvert A_{2}\rvert$. Choosing $\beta = \gamma$ and using H\"older's then Young's inequalities,
\begin{align*}
\lvert A_{2}\rvert \lesssim& \int_\Omega \eta^{\gamma-1} F(\Delta_h u\times \Der \eta)F_\ast(\Delta_h j_F^\gamma(\Der u))\\
\lesssim& \left(\int_{\Omega}F(\Delta_h u\otimes \Der \eta)^\gamma\right)^\frac 1 \gamma\left(\int_\Omega \eta^\gamma F_\ast \left(\Delta_h j_F^\gamma(\Der u) \right)^{\gamma^{\prime}}\right)^\frac 1 {\gamma^\prime}\\
\lesssim& C(\e)\lvert h\rvert^\gamma \|u\|_{\WW^{1,\gamma}(B_{\rho_2})}^\gamma+\e \int_\Omega \eta^\gamma F_\ast \left(\Delta_h j_F^\gamma(\Der u)\right)^{\gamma^\prime}.
\end{align*}
Since $\gamma^{\prime} \ge \tau^{\prime}$, up to replacing $F$ with $F_{\ast}$, $V_{p,q}$ with $V_{\gamma^{\prime},\tau^{\prime}}^{\ast}$, and $u$ with $j_{F}^{\gamma}(\nabla u)$, \eqref{e:1} reads
$$
F{\ast} \left( \Delta_{h} j_{F}^{\gamma}(\nabla u) \right)^{\gamma^{\prime}} \le F_{\ast} \left( \Delta_{h} V^{\ast}(j_F^\gamma(\Der u)) \right)^{\tau^{\prime}}.
$$
Hence,
\begin{equation} \label{e:A2c22}
\lvert A_{2}\rvert \lesssim \epsilon \int_{\Omega} \eta^{\gamma} F_{\ast} \left( \Delta_{h} V^{\ast}(j_F^\gamma(\Der u)) \right)^{\tau^{\prime}} + C(\epsilon) \lvert h\rvert^{\gamma} \|u\|_{\Ww^{1,\gamma}(B_{s})}.
\end{equation}
By writing $A_{1} \le \lvert A_{2}\rvert + \lvert II\rvert$ using respectively \eqref{e:A1vfunc}, \eqref{e:A2c22}, and \eqref{e:IIvfunc}  we obtain
\begin{align}\label{e:final4}
& \int_{\Omega} \eta^{\beta} F\left( \Delta_{h} V(\Der u)\right)^{\sigma} + \int_{\Omega} \eta^{\beta} F_{\ast} \left( \Delta_{h} V^{\ast}(j_F^\gamma(\Der u)) \right)^{\tau^{\prime}} \\
\nonumber & \lesssim \left(\lvert h\rvert^{\gamma}+\lvert h\rvert^2+\lvert h\rvert^{1+\alpha_\ast}\right)\left(\|u\|^{\gamma}_{W^{1,\gamma}(B_{s})} + \|f\|_{B^{\alpha_{\ast},\gamma^\prime}_\infty(B_s)}^{\gamma^\prime}\right).
\end{align}

Noting $\alpha_\ast+1= \gamma$, dividing by $\lvert h\rvert^{\gamma}$ implies
$$
  \| V(\Der u)  \|_{ \Bb^{\frac{\gamma^{\prime}}{\sigma},\sigma}_{\infty}(B_{r})} +  \| V^{\ast}(j_F^\gamma(\Der u)) \|_{\Bb^{\frac{\gamma^{\prime}}{\tau^{\prime}},\tau^{\prime}}_{\infty}(B_{r})} \lesssim_{(s-r)}\|u\|_{\Ww^{1,\gamma}(B_{s})}^\gamma+\|f\|_{B^{\tau,\gamma^{\prime}}_{\infty}(B_{s})}^{\gamma^\prime}.
$$
\textbf{Concluding the proof.}
One confirms \eqref{e:VFuncNorm}, and consequently \eqref{e:VFuncReg} using Case 1.1, 1.2, 2.1, and 2.2, respectively when $\alpha = \sigma^{\prime}, \gamma^{\prime}, \tau,$ and $\gamma$. To deduce \eqref{e:FuncReg} from \eqref{e:VFuncReg}, we apply Corollary \ref{c:VFuncEquiv} with $p=\gamma, q = \sigma$ to obtain
\begin{equation*}
\int_{B_{r}} F\left( \Delta_{h} V(\Der u) \right)^{\sigma} \gtrsim 
\begin{cases}
\int_{B_{r}} F(\Delta_{h} \nabla u)^{\gamma} & \gamma \ge \sigma \\
 \left( \int_{B_{r}} F \left( \Delta_{h} \nabla u \right)^{\gamma} \right)^{\frac{\sigma}{\gamma}} \left( \int_{B_{r}} (F(\nabla u) + F(\nabla u_{h}))^{\gamma} \right)^{1 - \frac{\sigma}{\gamma}} & \gamma \le \sigma.
\end{cases}
\end{equation*}

Similarly with $p = \gamma^{\prime}, q = \tau^{\prime}$, $F = F_{\ast}$ and $u = j_{F}^{\gamma}(\nabla u)$, Corollary \ref{c:VFuncEquiv} says

\begin{equation*}
\int_{B_{r}} F_{\ast} \left( \Delta_{h} V^{\ast}(j_F^\gamma(\Der u) \right)^{\tau^{\prime}} \gtrsim 
\begin{cases}
\int_{B_{r}} F_{\ast}\left(\Delta_{h} j_{F}^{\gamma}(\nabla u) \right)^{\gamma^{\prime}} & \gamma^{\prime} \ge \tau^{\prime} \\
\left( \int_{B_{r}} F_{\ast} \left( \Delta_{h} j_{F}^{\gamma}(\nabla u) \right)^{\gamma^{\prime}} \right)^{\frac{\tau^{\prime}}{\gamma^{\prime}}} \left\| \nabla u  \right\|_{\LL^{\gamma}(B_{s})}^{\frac{\gamma^{\prime}-\tau^{\prime}}{\gamma^{\prime}-1}} & \gamma^{\prime} \le \tau^{\prime}.
\end{cases}
\end{equation*}

Using these estimates to produce further lower bounds for \eqref{e:final1}, \eqref{e:final2}, \eqref{e:final3}, \eqref{e:final4} and then dividing by $\lvert h\rvert^{\alpha}$ verifies \eqref{e:FuncReg}.
\end{proof}

\bibliographystyle{alpha}
\bibdata{references}
\bibliography{references}

\end{document}